\documentclass[10pt]{article}
\usepackage{amssymb}
\usepackage{amssymb, amsfonts, amsthm, mathrsfs, hyperref}
\usepackage{graphicx, xcolor}
\usepackage[all]{xy}
\usepackage[leqno]{amsmath}

\theoremstyle{definition}
\newtheorem{thm}{Theorem}[section]

\newtheorem{cor}[thm]{Corollary}
\newtheorem{prop}[thm]{Proposition}
\newtheorem{dfn}[thm]{Definition}

\newtheorem{rmk}[thm]{Remark}
\theoremstyle{definition}
\newtheorem{ex}[thm]{Example}

\makeatletter

\def\id{\mathrm{id}}

\def\tr{\mathrm{Tr}}
\def\ccc{\mathbb{C}}
\def\bb{\mathbb{B}}
\def\dd{\mathbb{D}}
\def\hh{\mathbb{H}}
\def\zz{\mathbb{Z}}
\def\rr{\mathbb{R}}

\def\ff{\mathbb{F}}
\def\oo{\mathbb{O}}
\def\uu{\mathbb{U}}

\def\clo{\mathcal{O}}

\def\clw{\mathcal{W}}

\def\pt{\partial}
\def\bpt{\overline{\pt}}

\def\ud{\mathrm{d}}
\def\spann{\mathrm{span}}

\def\re{\mathrm{Re}}
\def\ima{\mathrm{Im}}

\makeatother

\title{Stable Forms, Vector Cross Products and Their Applications in Geometry}
\author{Teng Fei}

\begin{document}
\maketitle{}

\section{Introduction: Two Classical Problems}

\subsection{Stable Forms}

Let $V$ be an $n$-dimensional vector space over $\rr$ or $\ccc$. A $p$-form $\varphi\in\wedge^pV^*$ is called \emph{stable} if its orbit under the natural $GL(V)$-action is an open subset of $\wedge^pV^*$. According to \cite{hitchin2001}, it is classically known (see \cite{reichel1907}, \cite{schouten1931}, \cite{gurevich1935} and \cite{gurevich1948})\footnote{It appears that \'Elie Cartan \cite{cartan1894} and Engel \cite{engel1900}'s work definitely should be added into this list. For an amazing account of related math and history, see \cite{agricola2008}.} that stable forms occur only in the following cases:
\begin{itemize}
\item $p=1$, arbitrary $n\in\zz_+$.
\item $p=2$, arbitrary $n\in\zz_+$.
\item $p=3$, $n=6,7$ or 8.
\item The dual of each above situations. That is, if $p$-forms on $V$ have an open orbit , so do $(n-p)$-forms.
\end{itemize}

One can list the number of open orbits and write down the canonical form and stabilizer of each orbit case by case. The answer will be dependent on the field we are working with. We would like to omit the detailed description of each orbit at this moment and leave it to the next section.

\subsection{Vector Cross Products}

Let $W$ be an $m$-dimensional real vector space with an inner product $\langle\cdot,\cdot\rangle$ which is positive definite. Citing \cite{eckmann1942}, an \emph{$r$-fold cross product} on $W$ is a continuous map $X: W^r\to W$ such that \[\begin{split}&\langle X(w_1,\dots,w_r),w_i\rangle=0,\quad 1\leq i\leq r,\\ &X(w_1,\dots,w_r)\neq0\textrm{ as long as }w_1\wedge\dots\wedge w_r\neq0\end{split}\] for any $w_1,\dots,w_r\in W$.

Let $S_{m,k}$ be the Stiefel manifold, i.e., \[S_{m,k}=\{\textrm{orthonomal $k$-frames in }\rr^m\},\] where the standard metric on $\rr^m$ is equipped and if $k=m$, we further require that the $m$-frames are consistent with a given orientation. It is obvious that $S_{m,1}=S^{m-1}$ and $S_{m,m}=SO(m)$ and there is a natural fibration \[\pi^m_{n,k}:S_{m,n}\to S_{m,k}\] for $m\geq n\geq k$ by forgetting the last $n-k$ vectors of an orthonormal $n$-frame in $\rr^m$.

It is not hard to see that the existence of an $r$-fold vector cross product on $\rr^m$ is equivalent to the existence of a section of the fibration \[\pi^m_{r+1,r}:S_{m,r+1}\to S_{m,r}.\] Therefore one can ask the more general question that which fibrations $\pi^m_{n,k}:S_{m,n}\to S_{m,k}$ admit a section. This problem was completely solved by Eckmann \cite{eckmann1942} and Whitehead \cite{whitehead1962} based on the work of Borel \cite{borel1953} and Adams \cite{adams1960}. The answer turns out to be that $\pi^m_{n,k}:S_{m,n}\to S_{m,k}$ admits a section if and only if
\begin{itemize}
\item $k=1$, it reduces to the problem of vector fields on spheres which was solved by Adams.
\item $n=m$, $k=m-1$ for any $m\geq3$.
\item $m=7$, $n=3$ and $k=2$.
\item $m=8$, $n=4$ and $k=3$.
\end{itemize}
In any of the above cases with $n=k+1$, there exists a vector cross product which can be further made to be multilinear and skew-symmetric. Such vector cross products can be beautifully constructed using Hodge star operators and composition algebras.

We may extend the notion of vector cross product to vector spaces over more general fields as follows. Let $W$ be a finite dimensional vector space over a field $\ff$ whose characteristic is not 2 with a non-degenerate quadratic form $\langle\cdot,\cdot\rangle$\footnote{The quadratic form $\langle\cdot,\cdot\rangle$ can be indefinite when $\ff$ is $\rr$.}. An \emph{$r$-fold vector cross product} on $W$ is a multilinear map $X:W^r\to W$ such that \[\begin{split}
&\langle X(w_1,\dots,w_r),w_i\rangle=0,\quad 1\leq i\leq r,\\ &\langle X(w_1,\dots,w_r),X(w_1,\dots,w_r)\rangle=\det(\langle w_i,w_j\rangle)\end{split}\] for any $w_1,\dots,w_r\in W$.

Using purely algebraic method, Brown and Gray \cite{brown1967} classified all vector cross products. It turns out that the only possible numbers of $n$ and $r$ are exactly those given by Eckmann and Whitehead. In many cases, for instance $\ff=\rr$, all the isomorphism classes of vector cross products are also determined. We will say more about it in Section 3.

\subsection{Connection}

The main goal of this paper is to explore the connection between stable forms and vector cross products and its applications in differential geometry. We will be exclusively working with the real numbers. So from now on, all vector spaces are over $\rr$ without further mentioning. We show that stable forms and vector cross products are in some sense the two faces of the same coin. Using this dictionary, we give an explanation of the Riemannian and pseudo-Riemannian nature associated with the two open orbits of stable forms in dimension 6 and 7. This correspondence also leads to many new results in differential geometry. We prove a few refined results of Calabi \cite{calabi1958} and Gray \cite{gray1969} and their analogues in paracomplex geometry. In particular, new balanced 3-folds with trivial canonical bundle are identified. We are also able to construct many manifolds with $G_2$-structure of class $\clw_3$, and as a consequence, we obtain a variety of solutions to the Killing spinor equations with constant dilaton in dimension 7.

This paper is organized as follows. In Section 2, we will briefly review the theory of 3-forms in dimension 6 and 7. Section 3 is devoted to vector cross products on $\rr^7$ and $\rr^8$. In Section 4 we establish the correspondence between stable 3-forms in 6 and 7 dimensions and vector cross products on $\rr^7$ and $\rr^8$. Some geometrical consequences of this correspondence will be derived in Section 5. For a somewhat independent interest, several topics on applying Hitchin's nonlinear Hodge theory to the orbit $\clo_6^+$ will be discussed in Section 6.

\subsection{Acknowledgement}

The author would like to thank Prof. Shing-Tung Yau and Prof. Victor Guillemin for their constant encouragement and help. The author is also indebted with various helpful discussions with Bao-Sen Wu, Siu-Cheong Lau, Peter Smillie and Cheng-Long Yu. Thanks also goes to Prof. Simon Chiossi for pointing me to \cite{chiossi2002} and \cite{cabrera2006}.

\section{Classical Theory of Stable 3-forms}

In this section, we will review the theory of 3-forms in 6 and 7 dimensional vector spaces, mainly following \cite{hitchin2000}, \cite{hitchin2001}, \cite{hitchin2004} and \cite{herz1983}.

\subsection{3-forms in Dimension 6}

Let $V$ be an oriented 6-dimensional real vector space. For any 3-form $\Omega\in\wedge^3V^*$ on $V$, one can associate it with a linear map $K_\Omega:V\to V\otimes \wedge^6V^*$ by defining \[K_\Omega(v)=-\iota_v\Omega\wedge\Omega\in\wedge^5V^*\cong V\otimes\wedge^6V^*\] for any $v\in V$, where $\iota_v$ denotes the contraction with $v$.

One can define a quadratic function $\lambda:\wedge^3V^*\to(\wedge^6V^*)^2$ on the space of 3-forms by \[\lambda(\Omega)=\frac{1}{6}\tr(K^2_\Omega)\in(\wedge^6V^*)^2.\] We say $\lambda(\Omega)$ is \emph{positive} if it can be written as a square of an element of $\wedge^6V^*$ and we say it is \emph{negative} if $-\lambda(\Omega)$ is positive.

We have the following facts about stable 3-forms in 6 dimension:
\begin{thm}\label{thm21}
There are exactly two $GL(V)$-open orbits in $\wedge^3V^*$ which can be characterized by \[\clo_6^+:=\{\Omega\in\wedge^3V^*:\lambda(\Omega)>0\}\] and \[\clo_6^-:=\{\Omega\in\wedge^3V^*:\lambda(\Omega)<0\}\] respectively.

An element $\Omega^+$ of $\clo_6^+$ is of the form \[\Omega^+=e^1\wedge e^2\wedge e^3+e^5\wedge e^6\wedge e^7,\] where $\{e^1,e^2,e^3,e^5,e^6,e^7\}$ is a basis\footnote{This seemingly weird notation is designed for later use.} of $V^*$. Its stabilizer is isomorphic to $SL(3,\rr)\times SL(3,\rr)$.

An element $\Omega^-$ of $\clo_6^-$ is of the form \[\begin{split}\Omega^-&=e^1\wedge e^2\wedge e^3-e^1\wedge e^6\wedge e^7+e^2\wedge e^5\wedge e^7-e^3\wedge e^5\wedge e^6\\ &=\re(e^1+\sqrt{-1}e^5)(e^2+\sqrt{-1}e^6)(e^3+\sqrt{-1}e^7).\end{split}\] Its stabilizer is isomorphic to $SL(3,\ccc)$.
\end{thm}
\begin{proof}
See \cite{hitchin2000}, Section 2.
\end{proof}
\begin{rmk}
A more natural interpretation of the orbit $\clo_6^+$ should be the following. Let $\dd:=\rr[\tau]/(\tau^2-1)$ be the algebra of paracomplex numbers\footnote{Paracomplex numbers are also known as split-complex numbers, tessarines, double numbers, etc. For an amazing list of synonyms of paracomplex number, see \href{http://en.wikipedia.org/wiki/Split-complex_number\#Synonyms}{http://en.wikipedia.org/wiki/Split-complex\_number\#Synonyms}.}. A paracomplex number $u$ can be expressed in the form $u=a+b\tau$ where $a,b\in\rr$ are the so-called real and imaginary parts of $u$. After a suitable coordinate change, one finds out that an element $\Omega^+\in\clo_6^+$ is of the form \[\begin{split}\Omega^+&=e^1\wedge e^2\wedge e^3+e^1\wedge e^6\wedge e^7-e^2\wedge e^5\wedge e^7+e^3\wedge e^5\wedge e^6\\ &=\re(e^1+\tau e^5)(e^2+\tau e^6)(e^3+\tau e^7).\end{split}\] We will adopt this interpretation in the following sections.
\end{rmk}

\begin{thm}
For any $\Omega^-\in\clo_6^-$, we can define a complex structure $J_{\Omega^-}:V\to V$ by letting \[J_{\Omega^-}(v)=\dfrac{K_{\Omega^-}(v)}{\sqrt{-\lambda(\Omega^-)}}\] for any $v\in V$. Here $\sqrt{-\lambda(\Omega^-)}\in\wedge^6V^*$ is taken to be the 6-form consistent with the given orientation which squares to $-\lambda(\Omega^-)$. With respect to $J_{\Omega^-}$, $\Omega^-$ is the real part of a decomposable complex $(3,0)$-form.

For any $\Omega^+\in\clo_6^+$, we can define a paracomplex structure $L_{\Omega^+}$ on $V$, that is, $L_{\Omega^+}:V\to V$ satisfies $L_{\Omega^+}^2=\id$ and the $(\pm1)$-eigenspaces of $L_{\Omega^+}$ are both 3-dimensional. Like above, it is defined to be \[L_{\Omega^+}(v)=\dfrac{K_{\Omega^+}(v)}{\sqrt{\lambda(\Omega^+)}}.\] With respect to $L_{\Omega^+}$, $\Omega^+$ is the real part of a decomposable paracomplex $(3,0)$-form.
\end{thm}
\begin{proof}
See \cite{hitchin2000}, Section 2.
\end{proof}
\begin{rmk}
As proved in this theorem, for any $\Omega^-\in\clo_6^-$, there exists a decomposable complex 3-form $\alpha$ such that \[\Omega^-=\frac{1}{2}(\alpha+\bar{\alpha})=\re(\alpha) \textrm{ and }\alpha\wedge\bar{\alpha}\neq0.\] If we further require that $\sqrt{-1}\alpha\wedge\bar{\alpha}>0$, then the choice of $\alpha$ is unique. With such requirement, we use the notation \[\hat{\Omega}^-=\frac{\alpha-\bar{\alpha}}{2i}=\ima(\alpha).\] There is a natural $\ccc^*$-action on $\clo_6^-$ given by \[(\rho\exp(i\theta),\Omega^-)\mapsto\rho\cdot\re(e^{-i\theta}\alpha).\] It is not hard to see that two elements of $\clo_6^-$ lie in the same $\ccc^*$-orbit if and only if they define the same complex structure.

There is a similar story for the orbit $\clo_6^+$.
\end{rmk}

\subsection{3-forms in Dimension 7}

The main result of this subsection goes back to the work of \'Elie Cartan. For an exposition in English, we refer to Herz's paper \cite{herz1983}.

Let $W$ be a 7-dimensional real vector space. For a 3-form $\varphi\in\wedge^3W^*$, we can define a symmetric quadratic form $Q_\varphi$ on $W$ with value in $\wedge^7W^*$ by \[Q_\varphi(v,w)=\iota_v\varphi\wedge\iota_w\varphi\wedge\varphi\in\wedge^7W^*.\] After trivializing $\wedge^7W^*$, we get a genuine quadratic form on $W$ which we still call it $Q_\varphi$. Apparently the absolute value of the signature of $Q_\varphi$ does not depend on the choice of the trivialization. With further work, one can verify that the open $GL(W)$-orbit of a stable 3-form $\varphi$ is exactly given by the set of $\phi$'s such that $Q_\phi$ is nondegenerate and the absolute signature of $Q_\phi$ agrees with that of $Q_\varphi$. There are exactly two such orbits, which we denote by $\clo_7^-$ and $\clo_7^+$ respectively, with absolute signature 7 and 1 respectively.

A typical element $\varphi^-\in\clo_7^-$ can be written as \[\varphi^-=e^1\wedge e^2\wedge e^3-e^1\wedge e^6\wedge e^7+e^2\wedge e^5\wedge e^7-e^3\wedge e^5\wedge e^6+e^1\wedge e^4\wedge e^5+e^2\wedge e^4\wedge e^6+e^3\wedge e^4\wedge e^7,\] where $\{e^1,\dots,e^7\}$ is a basis of $W^*$. The identity component of the isotropy group of $\varphi^-$ is isomorphic to the compact 14-dimensional simple Lie group $G_2$, which in addition fixes a positive definite quadratic form on $W$.

On the other hand, the canonical form of an element $\varphi^+\in\clo_7^+$ is \[\varphi^+=e^1\wedge e^2\wedge e^3+e^1\wedge e^6\wedge e^7-e^2\wedge e^5\wedge e^7+e^3\wedge e^5\wedge e^6-e^1\wedge e^4\wedge e^5-e^2\wedge e^4\wedge e^6-e^3\wedge e^4\wedge e^7.\] The identity component of the isotropy group of $\varphi^+$ is isomorphic to the noncompact real form of $G_2$, which in addition fixes a quadratic form of signature $(3,4)$ on $W$.

\section{Vector Cross Products on $\rr^7$ and $\rr^8$}

The multiplication table of the 2-fold cross product on $\rr^7$ was already known to Cayley. However, it was not until 1965 that an explicit formula for 3-fold cross product on $\rr^8$ was found by Zvengrowski \cite{zvengrowski1965}. These are the only exceptional vector cross products according to the theorem of Eckmann\cite{eckmann1942}, Whitehead\cite{whitehead1962} and Brown-Gray\cite{brown1967}. In this section, we will recall the construction of these vector cross products and the classification theorem of Brown-Gray.

\subsection{Octonions and Split-Octonions}

Let $\hh$ be the division algebra of quaternion numbers. One can construct octonions and split-octonions from $\hh$ via the so-called Cayley-Dickson construction. Consider a pair of quaternion numbers which we write as $a+bl$ where $a,b\in\hh$ and $l$ is an indeterminate. We can define a product on such pairs of quaternion numbers by setting \[(a+bl)(c+dl)=(ac+l^2\bar{d}b)+(da+b\bar{c})l.\] If we set $l^2=-1$, then we get a non-associative division algebra $\oo$ known as the octonions. If we set $l^2=1$ instead, what we get is a non-associative composition algebra $\bb$ of the split-octonions. Let $\{e_0,e_1,e_2,e_3\}=\{1,i,j,k\}$ be the standard basis of $\hh$ and denote $e_{s+4}=e_sl$ for $s=0,1,2,3$, then in this way we get a basis $\{e_0,\dots,e_7\}$ of $\oo$ or $\bb$. For an element $x$ of either $\oo$ or $\bb$, we can express it as \[x=x_0e_0+\sum_{i=1}^7x_ie_i,\] where $x_0,\dots,x_7\in\rr$, and define its conjugate $\bar{x}$ by \[\bar{x}=x_0e_0-\sum_{i=1}^7x_ie_i.\] No surprise, we denote the real and imaginary parts of $x$ by $\re(x)=x_0e_0$ and $\ima(x)=\sum_{i=1}^7x_ie_i$ respectively. Furthermore, we can define a quadratic form $N(x)=x\bar{x}\in\rr$, making $\oo$ or $\bb$ a composition algebra, i.e., the following identity is satisfied \[N(xy)=N(x)N(y).\] This quadratic form $N$ defines an inner product on $\oo$ or $\bb$, with signature $(8,0)$ or $(4,4)$ respectively. We will use $\langle\cdot,\cdot\rangle$ to denote this particular inner product or its restriction on subspaces.

For later use, let us list some useful properties of $\oo$ and $\bb$:
\begin{itemize}
\item Both $\oo$ and $\bb$ are alternative algebras. As a consequence, any subalgebra generated by two elements (and their conjugates) is associative.
\item $\overline{xy}=\bar{y}\bar{x}$.
\item $x\bar{y}+y\bar{x}=2\langle x,y\rangle=2\langle\bar{x},\bar{y}\rangle$.
\item $x(\bar{y}z)+y(\bar{x}z)=2\langle x,y\rangle z$.
\end{itemize}

\subsection{Construction of Vector Cross Products}

We identify $\rr^7$ with the space of purely imaginary octonions $\ima(\oo)$ or purely imaginary split-octonions $\ima(\bb)$ with the standard inner product $\langle\cdot,\cdot\rangle$ of signature $(7,0)$ or $(3,4)$. In both cases, a 2-fold vector cross product $X$ can be defined on $\rr^7$ by \[X(a,b)=a\cdot b+\langle a,b\rangle e_0,\] where $\cdot$ is the multiplication in $\oo$ or $\bb$ in the previous subsection. It straightforward to check that this defines a cross product $X:\rr^7\times\rr^7\to\rr^7$ which is skew-symmetric. Furthermore, Theorem 4.1 of \cite{brown1967} showed that these are the only cross products on $\rr^7$ up to isomorphism. By an isomorphism of vector cross products between $(V,\langle\cdot,\cdot\rangle,X)$ and $(V',\langle\cdot,\cdot\rangle',X')$, we mean a linear isomorphism $\varphi:V\to V'$ preserving the inner product such that \[\varphi(X(v_1,\dots,v_r))=X'(\varphi(v_1),\dots,\varphi(v_r))\] for any $v_1,\dots,v_r\in V$.

For the $\rr^8$ scenario, we regard it as $\oo$ or $\bb$ with the inner product $\langle\cdot,\cdot\rangle$ of signature $(8,0)$ or $(4,4)$. In both cases, we can define two 3-fold vector cross products by \[\begin{split}&X_1(a,b,c)=-a(\bar{b}c)+\langle a, b\rangle c+\langle b,c\rangle a-\langle c,a\rangle b,\\ &X_2(a,b,c)=-(a\bar{b})c+\langle a, b\rangle c+\langle b,c\rangle a-\langle c,a\rangle b.\end{split}\] Note that the expression for $X_1$ is exactly the formula discovered by Zvengrowski \cite{zvengrowski1965}. It is routine to check that both $X_1$ and $X_2$ are totally skew-symmetric. Moreover, in Theorem 5.6 of \cite{brown1967}, Brown-Gray showed that $X_1$ and $X_2$ are not isomorphic to each other and they are the only isomorphism classes of 3-fold vector cross products.\footnote{In $\rr^8$, the only possible signatures for a 3-fold vector cross product to exist are $(8,0)$, $(4,4)$ and $(0,8)$. In each signature above, there are 2 non-isomorphic vector cross products $X_1$ and $X_2$. In this article, we only care about the signature $(8,0)$ and $(4,4)$ cases.}

It is clear from above description that 2-fold vector cross products on $\rr^7$ and 3-fold vector cross products on $\rr^8$ are closely related. Notice that $e_0$ lies in the center of $\oo$ or $\bb$ and the space of purely imaginary (split-)octonions is nothing but the orthogonal complement of $\rr\{e_0\}$, we see immediately that \[X(a,b)=X_1(e_0,a,b)=X_2(e_0,a,b).\]

\section{Relating Vector Cross Products and Stable Forms}

In this section, we establish the correspondence between vector cross products on $\rr^7$ and $\rr^8$ with sable 3-forms in dimension 6 and 7 in both directions.

\subsection{From Vector Cross Products to Stable Forms}

For simplicity, let us consider the Riemannian case first.

Let $(\oo,\langle\cdot,\cdot\rangle)\cong(\rr^8,\textrm{standard metric})$ be equipped with a 3-fold vector cross product $X':\wedge^3\oo\to\oo$ which can be either $X_1$ or $X_2$.

For any oriented line $l$ of $\oo$, there is a unique unit vector $a\in l$ compatible with the given orientation. We have the following result:

\begin{prop}
The 3-form $\varphi_l$ defined on $l^\perp$ by \[\varphi_l(x,y,z)=-\langle X'(x,y,z),a\rangle\] is a stable 3-form lying in the orbit $\clo_7^-$.
\end{prop}
\begin{proof}
Recall that the automorphism group of $(\oo,X',\langle\cdot,\cdot\rangle)$ is isomorphic to the Lie group $Spin(7)$ which acts transitively on $S^7$. Making use of such symmetry, we only have to prove the proposition for $a=e_0$. In this case (either $X'=X_1$ or $X'=X_2$), straightforward calculation shows that \[\varphi_l=e^1\wedge e^2\wedge e^3-e^1\wedge e^6\wedge e^7+e^2\wedge e^5\wedge e^7-e^3\wedge e^5\wedge e^6+e^1\wedge e^4\wedge e^5+e^2\wedge e^4\wedge e^6+e^3\wedge e^4\wedge e^7,\] which is exactly the canonical form of an element in $\clo_7^-$ described in Section 2.
\end{proof}

For any oriented 2-dimensional subspace $P$ of $\oo$, one can define a complex structure $J_P$ on $P^\perp$ by \[J_P(v)=-X'(a,b,v),\] where $\{a,b\}$ is an oriented orthonormal basis of $P$. Indeed, using the properties listed in Section 3, one can easily check that $J_P$ is a complex structure on $P^\perp$ compatible with the restriction of $\langle\cdot,\cdot\rangle$. In addition, it is independent of the choice of the oriented orthonormal basis $\{a,b\}$.
\begin{prop}
Fix a choice of $\{a,b\}$, the 3-form $\Omega_{P,a}$ on $P^\perp$ defined by \[\Omega_{P,a}(x,y,z)=-\langle X'(x,y,z),a\rangle\] is a stable 3-form lying in the orbit $\clo^-_6$. In addition, $J_{\Omega_{P,a}}=J_P$ and $\hat{\Omega}_{P,a}$ can be represented by \[\hat{\Omega}_{P,a}(x,y,z)=\begin{cases}\langle X'(x,y,z),b\rangle & \textrm{ if }X'=X_1\\ -\langle X'(x,y,z),b\rangle & \textrm{ if }X'=X_2\end{cases}.\]
\end{prop}
\begin{proof}
Again, we reduce the case to $P=\spann\{e_0,e_4\}$ by using the fact that $Spin(7)$ acts transitively on two dimensional subspaces of $\oo$, see, for instance, Proposition 2.1 of \cite{gray1969}. Without loss of generality, we may assume that $\{e_0,e_4\}$ is positively oriented. Straightforward calculation verifies the case $a=e_0$. The other cases can be derived from the following observation. If $X'=X_1$, we can define an $S^1$-action on oriented orthonormal basis of $P$ by \[(\exp(i\theta),\{a,b\})\mapsto (a\cos\theta-b\sin\theta,a\sin\theta+b\cos\theta).\] Then the map
$\{\textrm{oriented orthonormal basis of }P\}\to \clo^-_6$ given by \[\{a,b\}\mapsto\Omega_{P,a}\] is $S^1$-equivariant. If $X'=X_2$, we only need to replace the action by \[(\exp(i\theta),\{a,b\})\mapsto (a\cos\theta+b\sin\theta,-a\sin\theta+b\cos\theta).\]
\end{proof}

If we start with $(\rr^7,\textrm{standard metric})$ equipped with a 2-fold vector cross product $X$, we can simply define a stable 3-form $\varphi^-\in\clo^-_7$ on $\rr^7$ by setting \[\varphi^-(x,y,z)=\langle X(x,y),z\rangle.\] For any unit vector $b\in\rr^7$, the stable 3-form $\Omega^-_b=\varphi^-|_{b^\perp}$ lies in the orbit $\clo^-_6$.
\begin{rmk}
A unit vector $a\in\rr^8$ induces a 2-fold vector product $X$ on $a^\perp$ by the formula \[X(x,y)=X'(a,x,y).\] Such a reduction provides an alternative proof to Proposition 4.1 and 4.2.
\end{rmk}

The split case can be carried out in like manner. We only need to make sure that the line $l$ is space-like and the plane $P$ is Lorentzian. By replacing $Spin(7)$ by $Spin(3,4)$ and complex structures by paracomplex structures, we construct stable 3-forms in the orbits $\clo_7^+$ and $\clo^+_6$ respectively.

\subsection{From Stable Forms to Vector Cross Products}

Basically what we do here is to reverse the construction of the previous subsection.

Let $\Omega^-\in\clo^-_6$ be a stable 3-form on $V\cong\rr^6$. We know it defines a complex structure $J_{\Omega^-}$ on $V$. Let $\langle\cdot,\cdot\rangle$ be any positive definite inner product on $V$ compatible with $J_{\Omega^-}$, i.e., an inner product such that $J_{\Omega^-}$ is an orthogonal transformation. Such $\langle\cdot,\cdot\rangle$ always exists and we can define an Hermitian 2-form $\omega$ by \[\omega(x,y)=\langle J_{\Omega^-}(x),y\rangle.\] After scaling, we may even impose that \[\frac{1}{4}\Omega_-\wedge\hat{\Omega}_-=\frac{1}{6}\omega^3.\]

Consider the 7-dimensional vector space $W=V\oplus\rr$, where $\oplus$ is an orthogonal direct sum. Let $\beta$ be a unit 1-form on $W$ vanishing on $V$, then the 3-form \[\varphi^-=\Omega^--\beta\wedge\omega\] is a stable 3-form lying in the orbit $\clo^-_7$, which is compatible with the direct sum metric on $W$.

Given an element $\varphi^-\in\clo_7^-$, it is well-known that the group leaving $\varphi^-$ invariant is isomorphic to the compact $G_2$, therefore it determines a metric $\langle\cdot,\cdot\rangle$ on $W$. The 2-fold cross product $X$ on $W$ can be simply constructed by setting \[\langle X(x,y),z\rangle=\varphi^-(x,y,z).\] The construction of a 3-fold cross vector product on $W\oplus\rr$ from a 2-fold vector cross product on $W$ is classical, see Theorem 4.1 and 5.1 of \cite{brown1967} for instance.

For an element $\Omega^+\in\clo^+_6$, it defines a paracomplex structure $L_{\Omega^+}$ on $V$. A non-degenerate inner product $\langle\cdot,\cdot\rangle$ is said to be compatible with $L_{\Omega^+}$ if \[\langle L_{\Omega^+}(x),L_{\Omega^+}(y)\rangle=-\langle x,y\rangle\] for any $x,y\in V$. Given such an inner product, which always exist, we can define an Hermitian 2-form $\omega$ by \[\omega(x,y)=\langle L_{\Omega^+}(x),y\rangle.\] Like the Riemannian case, \[\varphi^+=\Omega^++\beta\wedge\omega\] is a stable 3-form on $W=V\oplus\rr$ lying in the orbit $\clo^+_7$. Notice that the $\rr$-component of $W$ in this case has a negative definite inner product. As proved by \'Elie Cartan \cite{cartan1894}, an element $\varphi^+$ of $\clo^+_7$ also determines a (3,4)-inner product on $W$. We can then construct 2-fold vector cross product and 3-fold vector cross product on $W$ and $W\oplus\rr$ respectively, exactly in the manner of Riemannian case.

\section{Geometric Consequences}

\subsection{Calabi and Gray's Constructions and Almost Paracomplex Structures}

Let $i:M\to\rr^7$ be an immersion where $M$ is an oriented 6-manifold. Calabi \cite{calabi1958} discovered that $M$ automatically admits an almost complex structure by the following construction. Let us identify $\rr^7$ with purely imaginary octonions and equip it with the 2-fold vector cross product $X$. For each $x\in M$, there is a unique unit normal vector $n_x\in T_{i(x)}\rr^7$ such that $\{n_x,T_xM\}$ is positively oriented. Using properties of $X$, one sees immediately that \[J_x=-X(n_x,\cdot):T_xM\to T_xM\] defines an almost complex structure on $M$. This almost complex structure $J$ is compatible with the pullback metric $g$ from $\rr^7$, and therefore one can cook up an Hermitian form $\omega$, which is shown to be coclosed in \cite{gray1969}. Calabi observed that $(M,g,J)$ actually defines an $SU(3)$-structure therefore its first Chern class (with respect to $J$) vanishes. This fact can be seen from applying our construction in Section 4.1 pointwise to get an explicit nowhere vanishing (3,0)-form on $M$ as well. Calabi also studied the integrability condition for $J$ and proved that if $\Sigma\subset\rr^3$ is a minimal surface, then the above construction for $i:\Sigma\times\rr^4\to\rr^3\times\rr^4\cong\rr^7$ gives an integrable complex structure which leads to very interesting compact complex 3-folds after taking quotients. Using our observations in Section 4, we can refine Calabi's result to the following:
\begin{thm}\label{thm51}
Let $i:M\to \rr^7$ be an immersion of an oriented 6-manifold. If the almost complex structure $J$ defined by Calabi is integrable, then the canonical bundle of $M$ is holomorphically trivial.
\end{thm}
\begin{proof}
The cross product on $\ima(\oo)\cong\rr^7$ defines a constant 3-form $\varphi$ by \[\varphi(u,v,w)=\langle X(u,v),w\rangle.\] From what we have seen in Section 4, we know that $\Omega_1=i^*\varphi$ is a closed stable 3-form on $M$ lying in the orbit $\clo_6^-$ pointwise. In addition, the almost complex structure associated with $\Omega_1$ is exactly $J$. By Theorem \ref{thm21}, we know that $\Omega_1$ is the real part of a complex decomposable (3,0)-form $\Omega=\Omega_1+\sqrt{-1}\Omega_2$.

On one hand, as $J$ is integrable, we know that \[\ud\Omega=\ud\Omega_1+\sqrt{-1}\ud\Omega_2=\sqrt{-1}\ud\Omega_2\] is a $(3,1)$-form. On the other hand, it is obvious that $\ud\Omega_2$ is real. The only possibility for this to happen is that $\ud\Omega_2=0$ and therefore $\Omega$ is holomorphic.
\end{proof}
\begin{rmk}\label{rmk52}
This theorem implies that those compact complex 3-folds diffeomorphic to $\Sigma_g\times T^4$ constructed in Section 5 of \cite{calabi1958} are actually non-K\"ahler Calabi-Yau's. They are examples of what is called \emph{special balanced} 3-fold later in this paper, which serve as candidates of internal space in heterotic strings. And it is very clear from the proof that we can replace $\rr^7$ by any 7-manifold with a $G_2$-structure whose fundamental 3-form $\varphi$ is closed, which are also known as $G_2$-structure of type $W_2$. The general relation of $SU(3)$-structures on hypersurfaces of manifolds with $G_2$-structure is studied in great detail in \cite{cabrera2006}. However, it seems that this simple theorem we proved here cannot be derived directly from Cabrera's paper.
\end{rmk}

Gray \cite{gray1969} further generalized Calabi's construction to the case that $i:M\to\bar{M}$ is an oriented (immersed) 6-dimensional submanifold, where $\bar{M}$ is a manifold with vector cross product of dimension 7 or 8. By definition, a manifold with vector cross product is a pseudo-Riemannian manifold $(\bar{M},\langle\cdot,\cdot\rangle)$ such that there is a vector cross product structure defined on each tangent space which varies smoothly. The existence of a vector cross product is equivalent to the condition that the structure group of $(\bar{M},\langle\cdot,\cdot\rangle)$ can be reduced to a subgroup of the automorphism group of corresponding vector cross product on a vector space. Manifolds with vector cross product must be orientable and from now on we always assume that an orientation is already chosen.

Under such assumption, Gray showed that if the restriction of $\langle\cdot,\cdot\rangle$ on the normal bundle $\nu$ of $M$ in $\bar{M}$ is definite\footnote{Gray\cite{gray1969} only considered the positive definite case.}, then there is a natural almost complex structure $J(=J_\nu)$ defined on $M$ which can be obtained by applying Proposition 4.2 pointwise. Applying the construction of $\Omega_{P,a}$ in Proposition 4.2 pointwise, we actually prove the following corollary:

\begin{cor}
If we further assume that $\nu$ as a real rank 2 bundle contains a trivial line bundle, then the first Chern class of $M$ with respect to $J$ is trivial.
\end{cor}
\begin{proof}
By assumption, we can pick up a unit normal vector field $a\in\Gamma(M,\nu)$ on $M$ to construct a 3-form $\Omega_{\nu,a}$ on $M$ whose complex form trivializes the (almost) canonical bundle of $M$.
\end{proof}

As an analogue of Calabi and Gray's construction, we can also use manifold with vector cross product to construct 6-dimensional almost paracomplex manifolds.

\begin{prop}\label{prop52}
Let $i:M\to(\bar{M},X,\langle\cdot,\cdot\rangle)$ be an immersion of an oriented 6-manifold $M$ into an 8-manifold $\bar{M}$ with a 3-fold vector cross product $X$ of split signature (4,4). If the restriction of $\langle\cdot,\cdot\rangle$ on the normal bundle $\nu$ is of signature $(1,1)$, then $M$ admits a natural almost paracomplex structure $L$ which is compatible with pullback metric from $\bar{M}$. Precisely, $L_p:T_pM\to T_pM$ is defined by \[L_p(v)=-X(a,b,v),\] where $\{a,b\}$ form an oriented orthogonal basis of $T_pM^\perp$ such that $\|a\|,\|b\|=\pm1$. Construction in Section 4 shows that $M$ admits a 3-form lying in the orbit $\clo_6^+$ pointwise. It reduces the structure group of $M$ to $SL(3,\rr)$, which is the paracomplex cousin of $SU(3)$.
\end{prop}
\begin{rmk}
This proposition generalizes the construction of almost paracomplex structure on the pseudosphere \[S^{3,3}:=\{x=(x_1,\dots,x_7)\in\ima(\bb):\langle x,x\rangle=x_1^2+\dots+x_3^2-x_4^2-\dots-x_7^2=-1\}\] by Libermann \cite{libermann1952}.
\end{rmk}

To explore the geometry of such almost paracomplex 6-manifolds, let us introduce a few useful concepts proposed by Gray.

Let $(\bar{M},\langle\cdot,\cdot\rangle)$ be a pseudo-Riemannian manifold with an $r$-fold vector cross product $X$. That is, $X$ is a $(1,r)$-tensor on $\bar{M}$ such that for any point $p\in\bar{M}$, the pair $(\langle\cdot,\cdot\rangle_p,X_p)$ is an $r$-fold vector cross product on $T_p\bar{M}$ in the usual sense. We can define an $(r+1)$-form $\mu$ on $\bar{M}$ by \[\mu(v_1,\dots,v_{r+1}):=\langle X(v_1,\dots,v_r),v_{r+1}\rangle.\] Let $\bar{\nabla}$ be the Levi-Civita connection associated with $\langle\cdot,\cdot\rangle$.

\begin{dfn}
The vector cross product $X$ is called
\begin{enumerate}
\item \emph{parallel}, if $\bar{\nabla}X=0$ or equivalently $\bar{\nabla}\mu=0$;
\item \emph{nearly parallel}, if $(\bar{\nabla}_{v}\mu)(v,v_1,\dots,v_r)=0$ for any vector fields $v,v_1,\dots,v_r$;
\item \emph{almost parallel}, if $\ud\mu=0$;
\item \emph{semi-parallel}, if $\delta\mu=0$, where $\delta$ is the codifferential.
\end{enumerate}
\end{dfn}

Now let $(\bar{M},X,\langle\cdot,\cdot\rangle$ be a 7-manifold with a 2-fold vector cross product or an 8-manifold with a 3-fold vector cross product. Let $M$ be an oriented 6-manifold and $i:M\to\bar{M}$ be an immersion such that the induced metric on the normal bundle $\nu$ of $M$ in $\bar{M}$ is definite. We have seen that there is a natural almost complex structure $J$ on $M$ which is compatible with the pullback metric $i^*\langle\cdot,\cdot\rangle$. Let $\nabla$ be the Levi-Civita connection on $M$ with respect to the pullback metric, and let $A$ be the configuration tensor, i.e., \[A(v,w)=\bar{\nabla}_vw-\nabla_vw\] for any $v,w$ vector fields on $M$ and \[A(v,n)=\textrm{tangent part of }\bar{\nabla}_vn\] for any tangent vector field $v$ and normal vector field $n$. Among many interesting results proved by Gray, we highlight the following two theorems:

\begin{thm}
(\cite{gray1969}, Theorem 6.5)\\
Suppose $X$ is parallel, then $J$ is integrable if and only if \[A(v,v)+A(Jv,Jv)=0\] for any vector field $v$ on $M$. As a corollary, integrability of $J$ implies that $M$ is a minimal submanifold of $\bar{M}$.
\end{thm}
\begin{rmk}
The condition $A(v,v)+A(Jv,Jv)=0$ actually indicates that $M$ is an austere submanifold of $\bar{M}$ in the sense of \cite{harvey1982}.
\end{rmk}
\begin{thm}
(\cite{gray1969}, Theorem 6.8)\\
If $X$ is semi-parallel, then $M$ is semi-K\"ahler, i.e., the Hermitian form $\omega$ associated to $J$ is coclosed, or equivalently in this case \[\ud(\omega^2)=0.\]
\end{thm}

We would like to prove the paracomplex analogues of these two theorems. This is something new because a paracomplex structure is not a 1-fold vector cross product in the setting that Gray was working with. No surprise, we have:

\begin{thm}\label{thm58}
Under the assumption of Proposition \ref{prop52}, if we further assume that the vector cross product $X$ is parallel, then the almost paracomplex structure $L$ is integrable if and only if the configuration tensor $A$ satisfies \[A(v,v)-A(Lv,Lv)=0\] for any tangent vector field $v$. In particular, integrability of $L$ implies that $M$ is a minimal submanifold of $\bar{M}$.
\end{thm}

\begin{proof}
The proof is along the lines of \cite{gray1969}, Section 6. So we will only sketch the proof and write down the corresponding formulae with signs modified.
\begin{itemize}
\item Step 1: $L$ is integrable if and only if \[(\nabla_x\omega)(y,z)+(\nabla_{Lx}\omega)(Ly,z)=0\] for any vector fields $x,y,z$. Where $\omega(x,y)=\langle Lx,y\rangle$ is the associated Hermitian form.
\item Step 2: We extend $L$ to the normal bundle $\nu$ by assigning $Ln$ to be the unit normal vector orthogonal to $n$ such that $\{n,Ln\}$ is positively oriented, assuming $n$ is also a unit normal vector. It is easy to verify that $L$ is a linear map that squares to identity. In addition, the following identities hold: \[\begin{split}&X(Lx,y,n)-LX(x,y,n)=\langle Lx,y\rangle n-\langle x,y\rangle Ln,\\ &X(Lx,Ly,n)-X(x,y,n)=2\langle Lx,y\rangle Ln.\end{split}\] Furthermore, exactly one of the following equations holds (depending on the kind of 3-fold cross vector product on $\bar{M}$), \[\begin{split}&LX(n,x,y)=X(Ln,x,y),\\ &LX(n,x,y)=-X(Ln,x,y)+2\langle Lx,y\rangle n.\end{split}\]
\item Step 3: Since $X$ is parallel, we have \[(\nabla_x\omega)(y,z)=\langle X(n,A(x,Ln)\pm LA(x,n),y),z\rangle.\] Using Step 1, integrability condition of $L$ is equivalent to that \[L(A(x,y)-A(Lx,Ly))\mp(A(x,Ly)-A(Lx,y))=0\] for any tangent vector fields $x,y$, which is equivalent to the condition given in the theorem.
\end{itemize}
\end{proof}

\begin{ex}
From the Cayley-Dickson construction, we can identify $\bb$ with $\hh\oplus\hh\cdot l\cong\rr^4\oplus\rr^4$, where the the restricted metric on $\hh$ and $\hh\cdot l$ are positive definite and negative definite respectively. Let $M=M_1\times M_2\subset\rr^4\times\rr^4$ be the product of two oriented hypersurfaces of $\rr^4$. It is easy to check that the condition of Proposition \ref{prop52} holds, therefore we get two almost paracomplex structures $L_1$ and $L_2$ on $M$. One can show that neither $L_1$ nor $L_2$ is integrable, unless both $M_1$ and $M_2$ are (part of) hyperplanes. In fact, $L_i$ switch the tangent vectors of $M_1$ and $M_2$. However, as a product space, $M_1\times M_2$ admits a standard paracomplex structure $L$. That is, $L$ acts on $M_1$ as identity and acts on $M_2$ as minus identity. Such $L$ is automatically integrable and it actually defines a para-Calabi-Yau structure, see Section 6.
\end{ex}

\begin{thm}
Under the assumption of Proposition \ref{prop52} and assuming that the vector cross product $X$ is semi-parallel, the Hermitian form $\omega$ is always coclosed.
\end{thm}
\begin{proof}
The proof is identical in the almost complex case, see \cite{gray1969}.
\end{proof}

From now on in this subsection, we shall consider the simplest case in the light of Proposition \ref{prop52}, i.e., $\bar{M}$ is the imaginary split-octonions $\ima(\bb)$ with the standard 2-fold cross product, $M$ is an oriented immersed hypersurface of $\ima(\bb)$ such that its normal is always time-like. We would like to construct interesting examples of paracomplex 6-manifolds of this form. Notice that such manifold must be non-compact.

\begin{rmk}
The almost Hermitian geometry of oriented hypersurfaces $M$ in $\ima(\oo)$ (or arbitrary 7-manifold $\bar{M}$ with a $G_2$-structure) and $\oo$ is well-studied, see \cite{calabi1958}, \cite{yano1964}, \cite{gray1966}, \cite{gray1969b}, \cite{cabrera2006} and the series papers by Banaru and Kirichenko, such as \cite{kirichenko1973, kirichenko1980, banaru1994, banaru2002} etc.. In particular, the condition for the almost Hermitian manifold $M$ to be of certain Gray-Hervella class \cite{gray1980} has been worked out. The representation theoretic classification of almost para-Hermitian manifolds is available in \cite{gadea1991}, so we may expect that most results on almost Hermitian geometry of $M$ also carry over to oriented hypersurfaces in $\ima(\bb)$ and $\bb$. However, we will not pursue this direction here.
\end{rmk}

There is another Cayley-Dickson construction of $\bb$ we shall describe. Let $\uu$ be the space of split-quaternion numbers, that is, a 4-dimensional real associative algebra with basis $1,i,j,k$ satisfying\[i^2=-1,~j^2=1,~k=ij=-ji.\] Any element of $\bb$ can be uniquely represented by $a+bl$ for $a,b\in\uu$, and the multiplication rule is exactly the formula given in Section 3, where $l^2$ can be either 1 or -1. With this in mind, we have the identification $\ima(\bb)=\ima(\uu)\oplus\uu\cdot l\cong\rr^{1,2}\oplus\rr^{2,2}$. Let $i:\Sigma\to\rr^{1,2}$ be an oriented immersed Lorentzian surface, that is, we require the induced metric on $\Sigma$ has signature $(1,1)$. Then by Proposition \ref{prop52}, the 6-manifold $M=\Sigma\times\rr^{2,2}$ has an almost paracomplex structure $L$ defined by 2-fold vector cross product on it. As an analogue of Theorem 6 of \cite{calabi1958}, we have
\begin{thm}
The almost paracomplex structure $L$ on $M=\Sigma\times\rr^{2,2}$ is integrable if and only if $\Sigma$ is an oriented minimal Lorentzian surface in $\rr^{1,2}$.
\end{thm}
\begin{rmk}
An oriented Lorentzian surface is the synonym of a para-Hermitian surface. Taking the standard para-Hermitian structure on $\rr^{2,2}\cong\uu\cong\dd^2$, we can produce a product para-Hermitian structure on $\Sigma\times\rr^{2,2}$. However, our paracomplex structure $L$ is twisted, that is to say, when restricted to the $\rr^{2,2}$ copy, its paracomplex structure varies according to the points on $\Sigma$.
\end{rmk}

To construct oriented Lorentzian minimal surfaces in $\rr^{1,2}$, one can use the paracomplex version of Weierstrass representation proposed by Konderak \cite{konderak2005}. Following Calabi's strategy in Section 5 of \cite{calabi1958}, we are able to construct non-flat 6-dimensional compact paracomplex manifold diffeomorphic to $T^6$. One may be doubtful about this procedure because that unlike Riemann surfaces, the only compact surface admitting a Lorentzian metric is the torus, so we do not have the paracomplex analogue of hyperelliptic curves in Calabi's original construction. However, there exist abundant para-holomorphic functions on a compact paracomplex manifold, for example on the Lorentzian torus. Therefore we still can produce para-holomorphic 1-forms satisfying right algebraic relations on a Lorentzian torus to make the Weierstrass representation work. Due to the same reason explained in Theorem \ref{thm51} and Remark \ref{rmk52}, the compact paracomplex 3-folds constructed here all have trivial canonical bundle, i.e., they are para-Calabi-Yau 3-folds (see Section 6).

There are not many known examples of compact paracomplex manifolds besides Cartesian products and certain solvmanifolds. A partial reason for this is that paracomplex projective spaces are noncompact. In addition, the paracomplex version of twistor construction, known as reflector spaces \cite{jensen1990}, fails to provide compact examples either. Hopefully the new compact paracomplex 6-manifolds described here may shed light to further understanding of paracomplex geometry.

\subsection{Construction of Manifolds with $G_2$ Structures}

We have already seen that oriented hypersurfaces in manifolds with structure group $G_2$ admits an $SU(3)$ structure. In this subsection, we would like to consider the converse, i.e., what kind of $G_2$ structure we can get from manifolds with structure group $SU(3)$. Here by $G_2$ we mean the compact 14-dimensional simple Lie group. The starting point is the following theorem
\begin{thm}
(\cite{gray1969}, Theorem 2.5)\\
Let $M$ be a 6-manifold with structure group $SU(3)$, i.e., an almost Hermitian manifold with $c_1=0$. If $p:\bar{M}\to M$ is a Riemannian submersion of a 7-manifold $\bar{M}$ onto $M$, then the structure group of $\bar{M}$ can be reduced to $G_2$.
\end{thm}

\begin{rmk}
Based on this theorem, the full correspondence between $SU(3)$ structures on $M$ and $G_2$ structures on $\bar{M}$ was first worked out in \cite{chiossi2002}. Some parts of our discussion below is covered by Theorem 5.1 of \cite{chiossi2002}, where more general $SU(3)$-structures are considered.
\end{rmk}

An important class of manifolds with $SU(3)$-structure is balanced 3-folds $M$ with trivial canonical bundle. To be precise, we mean that $M$ is a complex 3-fold admitting a nowhere vanishing holomorphic 3-form. In addition, $M$ is equipped with an Hermitian metric such that its Hermitian form $\omega$ is coclosed, i.e., \[\ud(\omega^2)=2\omega\ud\omega=0.\] Such manifolds play an important role in string theory. In particular, as a generalization of the usual K\"ahler Calabi-Yau's, they are candidates of internal space in heterotic strings. Namely, they are potential solutions to the Strominger system \cite{strominger1986}. Compact examples of such manifolds include K\"ahler Calabi-Yau's, those holomorphic $T^4$-bundle over Riemann surfaces of genus $g\geq3$ in Remark \ref{rmk52}, compact quotients of unimodular 3-dimensional complex Lie groups \cite{abbena1986}, torus bundle over $T^4$ and $K_3$ surface \cite{goldstein2004}, conifold transition of K\"ahler Calabi-Yau's \cite{fu2012} and branched double cover of twistor spaces \cite{lin2014}. However, solutions to the Strominger system are only found on K\"ahler Calabi-Yau's \cite{li2005}, torus bundles over $K_3$ \cite{fu2008} and (quotients of) unimodular complex Lie groups \cite{fei2015}. We leave the discussion of solutions to the Strominger system on $T^4$-bundle over Riemann surfaces of $g\geq3$ to a future paper.

Let $M$ be a balanced complex 3-fold with trivial canonical bundle, and let $(E,h)$ be a holomorphic line bundle over $M$ with an Hermitian metric $h$ and a metric-preserving connection $D$. Consider the unit circle bundle inside $E$, whose total space we shall denote by $\bar{M}$. As $D$ determines the splitting of the tangent space of $\bar{M}$ into the horizontal and vertical parts, it also defines a natural Riemannian metric on $\bar{M}$, making the projection $p:\bar{M}\to M$ to be a Riemannian submersion. Then by above theorem, we get a $G_2$ structure on $\bar{M}$, which will be the subject to study in this subsection.

Let $\xi$ be a local section of the unit circle bundle $p:\bar{M}\to M$ over an open subset $U\subset M$. We can define a fiber coordinate $\theta$ by expressing any point $p\in\bar{M}$ as $\exp(\sqrt{-1}\theta)\xi$. Also $\xi$ defines a connection 1-form $\sqrt{-1}\alpha$ on $M$ by \[D\xi=\sqrt{-1}\alpha\otimes\xi.\] Notice that $\alpha$ is real-valued, and the first Chern class of $E$ is given by \[c_1(E)=-\frac{\ud\alpha}{2\pi}.\] One can easily check that $\rho=\ud\theta+p^*\alpha$ is the globally defined 1-form on $\bar{M}$ such that \[\rho=0\textrm{ on }H\textrm{ and }\rho=\ud\theta\textrm{ on }V,\] where $H$ and $V$ are horizontal and vertical distributions determined by $D$ respectively.

Now let $\Omega=\Omega_1+\sqrt{-1}\Omega_2$ be a holomorphic volume form on $M$ and let $\omega$ be the balanced metric. By our assumption, \[\ud\Omega_1=\ud\Omega_2=\ud(\omega^2)=0.\] Then we can define a 3-form $\varphi$ on $\bar{M}$ by \[\varphi=p^*\Omega_1-\rho\wedge p^*\omega.\] From what we have seen, $\varphi$ lies in the orbit $\clo_7^-$ pointwise and therefore gives rise to a $G_2$ structure on $\bar{M}$. Moreover, calculation shows that \[*\varphi= p^*\Omega_2\wedge\rho-\frac{\omega^2}{2}\] is a stable 4-form on $\bar{M}$. Here $*$ is the Hodge star operator with respect to the Riemannian metric on $\bar{M}$.

We can draw two conclusions from this calculation:
\begin{itemize}
\item $\ud\varphi=0$ if and only if $c_1(E)=0$ as a 2-form and $\ud\omega=0$, which implies that $M$ is actually K\"ahler and $c_1(E)$ is a torsion class as an element in $H^2(M,\zz)$.
\item $\ud(*\varphi)=0$ because that $c_1(E)$ is a (1,1)-form.
\end{itemize}
\begin{rmk}
The $G_2$ structure determined by $\varphi$ also determines a metric on $\bar{M}$, which in general differs from the original metric on $\bar{M}$ determined by $D$. This fact distinguishes the Hodge star $*_\varphi$ defined by $\varphi$ from $*$ we used to compute $*\varphi$. However, when $M$ has constant dilaton, i.e., we may choose a holomorphic volume form $\Omega$ such that \[\frac{\sqrt{-1}}{8}\Omega\wedge\bar{\Omega}=\frac{1}{6}\omega^3,\] then the two metrics coincides. Such manifolds include K\"ahler Calabi-Yau's, our examples in Remark \ref{rmk52}, quotients of unimodular complex Lie groups and torus bundles over $T^4$ or $K_3$. As $*_\varphi\varphi$ plays an important role in the classification of $G_2$ structure by \cite{fernandez1982}, from now on in this section, we will assume that $M$ is a balanced 3-fold with trivial canonical bundle and constant dilaton. For later convenience, we will call $M$ a \emph{special balanced} 3-fold. It is conjectured that on any compact complex manifold with a fixed balanced metric, one can always find a balanced metric (insider the same suitable cohomology class) with prescribed normalized volume form. In particular for any balanced manifold with trivial canonical bundle, it always admits a special balanced structure. Geometrically this is equivalent to say that one can always find a balanced metric such that the holomorphic top form is parallel under the Strominger-Bismut connection. As far as the author knows, this conjecture is still open. For recent progress in this direction, we refer to the work of Sz\'ekelyhidi-Tosatti-Weinkove \cite{szekelyhidi2015}.
\end{rmk}

In conclusion, we have proved
\begin{prop}\label{prop515}(cf. \cite{chiossi2002}, Theorem 5.1)\\
Let $M$ be a special balanced 3-fold, and let $\bar{M}$ be the total space of any principal $U(1)$-bundle over $M$ whose real coefficient first Chern class can be represented by a $(1,1)$-form $F$. It is well-known that any such bundle comes from a holomorphic line bundle over $M$ with an Hermitian metric. We can always choose a $U(1)$-connection $D$ on $\bar{M}$ such that its curvature form is exactly $-2\pi\sqrt{-1}F$. Then by above construction using connection $D$, we cook up a $G_2$-structure $\varphi$ on $\bar{M}$ which is always semi-parallel, i.e., the 3-form $\varphi$ lying in $\clo^-_7$ pointwise satisfies \[\delta\varphi=0,\] where $\delta$ is the codifferential with respect to the metric determined by $\varphi$.

Moreover, if $\varphi$ defines a parallel $G_2$-structure, in other words, $\bar{M}$ has holonomy group contained in $G_2$, then $M$ must be a K\"ahler Calabi-Yau and $F=0$. In such case, the holonomy group of $\bar{M}$ is actually a subgroup of $SU(3)$.
\end{prop}

Let us assume now that we are in the scenario that $\delta\varphi=0$, or equivalently $*\varphi$ is a closed 4-form. In the sense of calibration geometry \cite{harvey1982}, we know that $*\varphi$ defines a calibration on $\bar{M}$. So we can talk about submanifolds calibrated by $*\varphi$ known as coassociative submanifolds. It is well-known that any such 4-manifold minimizes its volume in its homology class. Let $K$ be a 3-dimensional submanifold of $M$ and let $\bar{K}$ be the total space of the circle bundle $\bar{M}\to M$ restricted to $K$, which is a 4-dimensional submanifold of $\bar{M}$. It is easy to observe the following

\begin{prop}
$\bar{K}$ is a coassociative submanifold of $\bar{M}$ if and only if $K$ is a special Lagrangian submanifold of $M$ in the sense that $K$ is calibrated by $\Omega_2=\re(\Omega)$.
\end{prop}

The SYZ picture \cite{strominger1996} implies that $M$ admits a special Lagrangian $T^3$ fibration (at least in the case $M$ is K\"ahler). This naturally leads to a coassociative $T^4$ fibration of $\bar{M}$. Such a fibration was predicted by Gukov-Yau-Zaslow in \cite{gukov2003} from $M$-theory, which fits perfectly in the reduction/duality picture.

In Fern\'andez and Gray's classification of Riemannian manifold with $G_2$-structure \cite{fernandez1982}, semi-parallel $G_2$-structure was called class $\mathcal{ST}$ or class $\mathcal{W}_1\oplus\mathcal{W}_3$. Actually what Fern\'andez and Gray did is the following. Let $N$ be a Riemannian manifold with $G_2$-structure $\varphi$, $\nabla$ its associated Levi-Civita connection, one can show that $\nabla\varphi$ always satisfies the symmetry \[(\nabla_x\varphi)(y,z,X(y,z))=0\] for any vector fields $x,y,z$, where $X$ is the 2-fold vector cross product on $N$ determined by $\varphi$. Let $W$ be the space of tensors on $N$ that has exactly the symmetry as $\nabla\varphi$. The natural $G_2$-representation on $W$ splits into 4 irreducible components \[W=W_1\oplus W_2\oplus W_3\oplus W_4.\] Therefore according to what nontrivial components $\nabla\varphi$ has, one can classify $G_2$-structures in 16 classes. For the complete list of these classes, see Table I of \cite{fernandez1982}, here we list several important classes only:
\begin{itemize}
\item Parallel, a.k.a. class $\mathcal{T}$: $\nabla\varphi=0$.\\ Such manifold has holonomy group contained in $G_2$ and must be Ricci-flat. The first example of complete (noncompact) Riemannian manifold with holonomy exactly $G_2$ was constructed by Bryant and Salamon \cite{bryant1989}. First compact example was due to Joyce \cite{joyce1996}.
\item Nearly parallel, a.k.a. class $\mathcal{NT}$ or class $\clw_1$: $\ud\varphi=4\nabla\varphi$ or equivalently $\nabla\varphi=\dfrac{1}{168}\langle\nabla\varphi,*\varphi\rangle*\varphi$.\\ Such manifolds are also known as manifolds with weak holonomy $G_2$ in the sense of Gray \cite{gray1971} through the work of B\"ar \cite{bar1993}. A standard example of manifolds of this type is the round 7-sphere $S^7$ \cite{gray1969}.
\item Almost parallel, a.k.a. class $\mathcal{AT}$ or class $\clw_2$: $\ud\varphi=0$. \\ Such structures were found by Fern\'andez and Gray \cite{fernandez1982} on $TM\times\rr$ for any non-flat Riemannian 3-manifold $M$. As far as the author knows, the first compact example is due to Fern\'andez \cite{fernandez1987}.
\item Class $\clw_3$: $\delta\varphi=\langle\nabla\varphi,*\varphi\rangle=0$.\\ We will discuss more about this class later.
\item Semi-parallel, a.k.a. class $\mathcal{ST}$ or class $\clw_1\oplus\clw_3$: $\delta\varphi=0$.\\ Both compact and noncompact examples are given in Gray \cite{gray1969} where he showed that the natural $G_2$-structure on any oriented hypersurface in $\rr^8$ is semi-parallel.
\end{itemize}

Proposition \ref{prop515} enables us to construct $G_2$-structures in class $\clw_1\oplus\clw_3$. A natural question to ask is if we can obtain $G_2$-structures in class $\clw_1$ or $\clw_3$. From the characterization of these classes, we know that we have to first compute $\nabla\varphi$ before answering this question.

Let us do some local calculation based on our previous notations. Let $\xi$ be a local section of $p:\bar{M}\to M$ over an open subset $U\subset M$. This local section $\xi$ also determines a trivialization $\psi:p^{-1}(U)\cong U\times S^1$ over $U$. Let $\{e_1,e_2,e_3,e_5,e_6,e_7\}$ be an orthonormal frame over $U$ (or a smaller open subset) with dual basis $\{e^1,e^2,e^3,e^5,e^6,e^7\}$. If we denote by $\bar{e}_i$ the horizontal lifts of $e_i$, then it follows directly that $\{\bar{e}_1,\bar{e}_2,\bar{e}_3,\bar{e}_5,\bar{e}_6,\bar{e}_7,\pt_\theta\}$ is an orthonormal frame over $p^{-1}(U)$ and their dual basis is simply $\{p^*e^1,p^*e^2,p^*e^3,p^*e^5,p^*e^6,p^*e^7,\rho\}$. For simplicity of notations, we will omit $p^*$ when there is no ambiguity. We also use the convention that the Levi-Civita connections of $\bar{M}$ and $M$ are denoted by $\bar{\nabla}$ and $\nabla$ respectively.

It is not hard to see that \[\bar{e}_i=e_i-\alpha(e_i)\pt_\theta\] as a vector field on $U\times S^1$ under the above trivialization $\psi$. As a result, we get \[\begin{split}[\bar{e}_i,\bar{e}_j]&=[e_i,e_j]-e_i(\alpha(e_j))\pt_\theta+e_j(\alpha(e_i))\pt_\theta\\ &=\overline{[e_i,e_j]}-\ud\alpha(e_i,e_j)\pt_\theta,\\ [\bar{e}_i,\pt_\theta]&=[\pt_\theta,\pt_\theta]=0.\end{split}\]

As $M$ is special balanced, we can pick up $\{e_1,e_2,e_3,e_5,e_6,e_7\}$ appropriately such that \[\begin{split}\omega&=e^1\wedge e^5+e^2\wedge e^6+e^3\wedge e^7,\\ \Omega_1&=e^1\wedge e^2\wedge e^3-e^1\wedge e^6\wedge e^7+e^2\wedge e^5\wedge e^7-e^3\wedge e^5\wedge e^6,\\ \Omega_2&=e^1\wedge e^2\wedge e^7-e^1\wedge e^3\wedge e^6+e^2\wedge e^3\wedge e^5-e^5\wedge e^6\wedge e^7.\end{split}\]

By our assumption, $\ud\alpha=-2\pi c_1(E)$ is a (1,1)-form, so the only non-vanishing components of $\ud\alpha$ are $\ud\alpha(e_1,e_5)$, $\ud\alpha(e_2,e_6)$ and $\ud\alpha(e_3,e_7)$.

By applying Koszul's formula to compute $\bar{\nabla}$ and using the fact that $p:\bar{M}\to M$ is a Riemannian submersion, we can prove that \[\begin{split}\bar{\nabla}_{\bar{e}_i}\bar{e}_j&=\Gamma_{ij}^k\bar{e}_k-\frac{1}{2}\ud\alpha(e_i,e_j)\pt_\theta,\\ \bar{\nabla}_{\bar{e}_i}\pt_\theta&=\frac{1}{2}\sum_j\ud\alpha(e_i,e_j)\bar{e}_j,\\ \bar{\nabla}_{\pt_\theta}\bar{e}_i&=\frac{1}{2}\sum_j\ud\alpha(e_i,e_j)\bar{e}_j,\\ \bar{\nabla}_{\pt_\theta}\pt_\theta&=0.\end{split}\] It follows that \[\begin{split}\bar{\nabla}_{\bar{e}_i}e^j&=\nabla_{e_i}e^j-\frac{1}{2}\ud\alpha(e_i,e_j)\rho,\\ \bar{\nabla}_{\bar{e}_i}\rho&=\frac{1}{2}\sum_j\ud\alpha(e_i,e_j)e^j,\\ \bar{\nabla}_{\pt_\theta}e^i&=\frac{1}{2}\sum_j\ud\alpha(e_i,e_j)e^j,\\ \bar{\nabla}_{\pt_\theta}\rho&=0.\end{split}\]

As $\varphi=\Omega_1-\rho\wedge\omega$, we get further compute \[\begin{split}\bar{\nabla}_{\pt_\theta}\varphi&=\frac{1}{2}\langle\ud\alpha,\omega\rangle\Omega_2= -\frac{1}{2}\langle\ud\alpha,\omega\rangle\iota_{\pt_\theta}(*\varphi),\\ \bar{\nabla}_{\bar{e}_i}\varphi&=\frac{1}{2}\ud\alpha(e_i,e_{i+3})\iota_{\bar{e}_i}(*\varphi)+\nabla_{e_i}\Omega_1- \rho\wedge\nabla_{e_i}\omega\textrm{ for }i=1,2,3,\\ \bar{\nabla}_{\bar{e}_i}\varphi&=\frac{1}{2}\ud\alpha(e_{i-3},e_i)\iota_{\bar{e}_i}(*\varphi)+\nabla_{e_i}\Omega_1- \rho\wedge\nabla_{e_i}\omega\textrm{ for }i=4,5,6.\end{split}\]

\begin{thm}\label{thm516}(cf. \cite{chiossi2002}, Theorem 5.1)\\
Under the setting of Proposition \ref{prop515}, we get a $G_2$-structure $\varphi$ of class $\clw_1\oplus\clw_3$ on $\bar{M}$. If $\varphi$ actually lies in class $\clw_1$, then $\varphi$ is parallel, so the holonomy group of $\bar{M}$ is contained in $SU(3)$. On the other hand, $\varphi$ lies in class $\clw_3$ if and only if $c_1(E)=-\dfrac{1}{2\pi}\ud\alpha$ is a primitive $(1,1)$-form, i.e., \[\langle\ud\alpha,\omega\rangle=0,\] or equivalently \[\ud\alpha\wedge\omega^2=0.\]
\end{thm}
\begin{proof}
Plug in our previous calculation, we see that \[\begin{split}\langle\bar{\nabla}\varphi,*\varphi\rangle&=\langle\bar{\nabla}_{\pt_\theta}\varphi, \iota_{\pt_\theta}(*\varphi)\rangle+\sum_i\langle\bar{\nabla}_{\bar{e}_i}\varphi,\iota_{\bar{e}_i}(*\varphi)\rangle\\ &=\frac{1}{2}\langle\ud\alpha,\omega\rangle\|\iota_{\pt_\theta}(*\varphi)\|^2+ \sum_i\langle\nabla_{e_i}\Omega_1-\rho\wedge\nabla_{e_i}\omega,\iota_{\bar{e}_i}(*\varphi)\rangle.\end{split}\]
Here we have used the facts that \[\|\iota_{\pt_\theta}(*\varphi)\|=\|\iota_{\bar{e}_i}(*\varphi)\|\] and \[\langle\ud\alpha,\omega\rangle=\ud\alpha(e_1,e_5)+\ud\alpha(e_2,e_6)+\ud\alpha(e_3,e_7).\] As $*\varphi=\Omega_2\wedge\rho-\dfrac{1}{2}\omega^2$, we have \[\iota_{\bar{e}_i}(*\varphi)=\rho\wedge\iota_{e_i}\Omega_2-(\iota_{e_i}\frac{\omega^2}{2}),\] we conclude that \[\langle\nabla_{e_i}\Omega_1-\rho\wedge\nabla_{e_i}\omega,\iota_{\bar{e}_i}(*\varphi)\rangle=-\langle\nabla_{e_i}\Omega_1, \iota_{e_i}\frac{\omega^2}{2}\rangle-\langle\iota_{e_i}\Omega_2,\nabla_{e_i}\omega\rangle.\] Multiplying the volume form, we get \[\begin{split}\left(\langle\nabla_{e_i}\Omega_1, \iota_{e_i}\frac{\omega^2}{2}\rangle+\langle\iota_{e_i}\Omega_2,\nabla_{e_i}\omega\rangle\right)\frac{\omega^3}{6}&= \nabla_{e_i}\Omega_1\wedge*(\iota_{e_i}\frac{\omega^2}{2})+\nabla_{e_i}\omega\wedge*(\iota_{e_i}\Omega_2)\\ &=\nabla_{e_i}\Omega_1\wedge\omega\wedge e^i+\nabla_{e_i}\omega\wedge\Omega_1\wedge e^i\\ &=0.\end{split}\] This is because $\omega\wedge\Omega_1=0$ hence $\nabla(\omega\wedge\Omega_1)=0$. Therefore we conclude that \[\langle\nabla\varphi,*\varphi\rangle=\frac{1}{2}\langle\ud\alpha,\omega\rangle\|\iota_{\pt_\theta(*\varphi)}\|^2,\] and the second part of our theorem follows directly.

For the first part, assuming $\varphi$ is of class $\clw_1$, then we know that $\ud\varphi=4\bar{\nabla}\varphi$, hence in particular \[\iota_{\pt_\theta}\ud\varphi=4\bar{\nabla}_{\pt_\theta}\varphi.\] Recall that $\ud\varphi=-\ud\rho\wedge\omega+\rho\wedge\ud\omega$, so from our previous calculation, we get \[\ud\omega=\iota_{\pt_\theta}(-\ud\rho\wedge\omega+\rho\wedge\ud\omega)=4\bar{\nabla}_{\pt_\theta}\varphi= 2\langle\ud\alpha,\omega\rangle\Omega_2.\] Notice that $\ud\omega$ is of type $(2,1)+(1,2)$ while $\Omega_2$ is of type $(3,0)+(0,3)$, this can happen only if $\ud\omega=0$ and $\langle\ud\alpha,\omega\rangle=0$. By our formal calculation, it further implies that $\langle\bar{\nabla}\varphi,*\varphi\rangle=0$, and hence $\nabla\varphi=0$ by Fern\'andez and Gray's characterization of class $\clw_1$. So $\varphi$ is parallel.
\end{proof}
\begin{rmk}
In \cite{cabrera1996}, Cabrera gave a much simpler characterization of all 16 classes of $G_2$-structures which involves differential forms only. In particular, in order to justify that $\bar{M}$ is of class $\clw_3$, we only have to show that $\delta\varphi=0$ and $\ud\varphi\wedge\varphi=0$. This provides a much simpler proof of our theorem. However, here we prefer working out the Levi-Civita connection explicitly for future use.
\end{rmk}

Let us consider the case that $E$ is a holomorphic line bundle over $M$. The condition (as differential form) that \[c_1(E)\wedge\omega^2=0\] implies that $E$ has degree 0 with respect to the polarization $\omega$. As line bundles are stable and balanced metrics are Gauduchon, i.e., $\pt\bpt(\omega^2)=0$, we can apply Li-Yau's non-K\"ahler version of Uhlenbeck-Yau theorem \cite{li1987} to conclude that this condition is equivalent to that the same equation holds as cohomology class. This observation enables us to obtain abundant examples of 7-manifolds with $G_2$-structure of class $\clw_3$. For example, if $M$ has Picard number greater than 1 and $[\omega^2]$ is a rational cohomology class, we can always find $E$ such that $c_1(E)\wedge\omega^2=0$ holds. In particular, we have the following example where $M$ is non-K\"ahler.

\begin{ex}
Let $S$ be a K3 surface of Picard number $\rho$ at least 4 with a Calabi-Yau structure such that its K\"ahler class is rational. Our assumption allows us to find two integral anti self-dual (1,1)-classes $\dfrac{\omega_P}{2\pi}$ and $\dfrac{\omega_Q}{2\pi}$ which are linearly independent. Applying the construction in \cite{goldstein2004}, we get a holomorphic $T^2$-bundle over $S$ whose total space $S^{P,Q}$ has a natural special balanced (non-K\"ahler) structure. As $\omega_P$ and $\omega_Q$ are linearly independent, from a Gysin sequence argument we know that the pull-back map $H^2(S)\to H^2(S^{P,Q})$ is surjective. It indicates that the balanced class on $S^{P,Q}$ is a multiple of an integral class. In addition, the Picard number of $S^{P,Q}$ is at least $\rho-2\geq2$, so we can apply Theorem \ref{thm516} to construct many 7-manifolds with $G_2$-structure of class $\clw_3$ in this way.
\end{ex}

\begin{rmk}
The first non-parallel examples of $G_2$-structures of class $\clw_3$ were found on minimal oriented hypersurfaces in $\rr^8$ by Fern\'andez and Gray \cite{fernandez1982}. Obviously they cannot be compact. The only compact example the author can find in literature is given in \cite{cabrera1996b}, where Cabrera, Monar and Swann proved that there exist $G_2$-structures of class $\clw_3$ on certain Aloff-Wallach manifolds, which are homogeneous spaces of the form $SU(3)/U(1)$. It would also be interesting to relate our construction to geometric quantization and Sasakian geometry.
\end{rmk}

As $p:\bar{M}\to M$ is a Riemannian submersion, we can relate the Riemannian curvature of $\bar{M}$ to the Riemannian curvature on $M$ to get:
\begin{prop}Under the assumption of Proposition \ref{prop515}, we have
\begin{enumerate}
\item The scalar curvature of $\bar{M}$ as a function is the pull-back of the scalar curvature of $M$.
\item $\bar{M}$ is Ricci-flat if and only if $M$ is Ricci-flat and $F=0$. The Ricci-flatness of $M$ implies that $M$ is a K\"ahler Calabi-Yau.
\end{enumerate}
\end{prop}
\begin{proof}
The relation of curvatures of $\bar{M}$ and $M$ can be easily proved by using the covariant derivative formulae we computed. The only thing left is to show that a Ricci-flat special balanced manifold is K\"ahler. This can be proved by using Formula (4.3) of \cite{liu2014}, which says \[Ric_H=\Theta^{(1)}-\frac{1}{2}\left(\pt\pt^*\omega+\bpt\bpt^*\omega\right)-\frac{\sqrt{-1}}{4}T\circ\bar{T},\] where $Ric_H$ is the Ricci form associated with the Riemannian Ricci-tensor, $\Theta^{(1)}$ is known as the first Ricci-Chern curvature, and $T$ is the torsion tensor of the Chern connection. Ricci-flatness says that the left hand side vanishes. Special balanced condition implies that the only non-trivial term on right hand side is the last term. By taking trace of $T\circ\bar{T}$ (Lemma 3.5 of \cite{liu2014}), we conclude that $T$=0, i.e. $\omega$ is K\"ahler.
\end{proof}
\begin{rmk}
As Bonan \cite{bonan1966} showed that $G_2$-manifolds are Ricci-flat, part (b) of the above proposition can be viewed as a stronger version of the second part of Proposition \ref{prop515}.
\end{rmk}

Now let $\bar{M}$ as constructed above be a 7-manifold with $G_2$-structure of type $\clw_3$, i.e., the fundamental 3-form $\varphi$ satisfies \[\delta\varphi=0,\quad\ud\varphi\wedge\varphi=0.\] By Theorem 4.7 of \cite{friedrich2002}, there exists a unique affine connection $\nabla^s$ on $\bar{M}$ with totally skew symmetric torsion $T$ such that \[\nabla^s\varphi=0.\] Moreover, by \cite{friedrich2002}, Theorem 4.8, the torsion 3-form $T$ in this case is actually given by \[T=-*d\varphi,\] which is automatically coclosed. These manifolds with $G_2$-structure of type $\clw_3$ are very important in the sense that they provide constant dilaton solutions to the Killing spinor equations in string theory, see \cite{friedrich2003}, Theorem 1.2.

\section{Nonlinear Hodge Theory}

In this section, we study the geometry associated with the orbit $\clo_6^+$ via nonlinear Hodge theory of Hitchin. Let us first recall Hitchin's functional from \cite{hitchin2000}.

Let $M$ be a compact oriented 6-manifold. Hitchin's functional $\Phi$ is a functional defined on 3-forms on $M$ given by \[\Phi(\Omega)=\int_M\sqrt{|\lambda(\Omega)|},\] where $\lambda(\Omega)\in(\Omega^6(M))^2$ is given in Section 2 and the square root is taken to be consistent with the fixed orientation. It is obvious that $\Phi$ is invariant under orientation preserving diffeomorphisms.

Suppose that $\Omega^-$ is a closed 3-form on $M$ such that it lies in the orbit $\clo_6^-$ pointwise. As $\Omega^-$ defines a de Rham cohomology class $[\Omega^-]\in H^3(M;\rr)$, one can consider the variation problem associated with $\Phi$ within the cohomology class $[\Omega^-]$. This is known as nonlinear Hodge theory in \cite{hitchin2000}, where Hitchin proved the following theorem:

\begin{thm}(\cite{hitchin2000}, Theorem 13)\\
Under the above assumption, $\Omega^-$ is a critical point of $\Phi$ in the cohomology class $[\Omega^-]$ if and only if $\Omega^-$ is the real part of a holomorphic (3,0)-form with respect to the complex structure $J_{\Omega^-}$.
\end{thm}

In other words, critical points of $\Phi$ are exactly complex 3-folds with trivial canonical bundle. And from this viewpoint, Hitchin gave a simple proof of unobstructedness of deformation of Calabi-Yau 3-folds.

A natural question to ask is that what are the critical points of $\Phi$ in the orbit $\clo_6^+$. With our paracomplex interpretation of $\clo_6^+$, it is straightforward to prove
\begin{thm}\label{thm62}
Let $M$ be a closed oriented 6-manifold and $\Omega^+\in\Omega^3(M)$ is a closed 3-form on $M$ such that $\Omega^+$ lies in the orbit $\clo_6^+$ pointwise. Then $\Omega^+$ is a critical point of $\Phi$ if and only if $\Omega^+$ can be written as the real part of a nowhere null para-holomorphic (3,0)-form with respect to the paracomplex structure $L_{\Omega^+}$.
\end{thm}
\begin{rmk}
There is a vast but ununified literature on paracomplex geometry. For a general knowledge of paracomplex manifolds, we recommend \cite{cruceanu1996}, \cite{etayo2006}, \cite{alekseevsky2009} and \cite{rossi2013}. It is evident that critical points of $\Phi$ in Theorem \ref{thm62} are exactly compact paracomplex 3-folds with trivial canonical bundle, which will be called para-Calabi-Yau's. If in addition, $M$ admits a compatible para-Hermitian metric whose associated Hermitian form is closed, then we will call $M$ a K\"ahler para-Calabi-Yau. It seems that such manifolds were first studied in \cite{harvey2012}, where split special Lagrangian submanifolds of para-Calabi-Yau's were investigated from a calibrated geometric point of view.
\end{rmk}

Now we would like to give more examples of compact para-Calabi-Yau's and discuss their geometric properties and potential relations to string theory. Let us first give an equivalent characterization of para-Calabi-Yau's.

\begin{prop}(cf. \cite{harvey2012}, Proposition 11.2')\\
Let $M$ be a smooth manifold of dimension $2n$. A para-Calabi-Yau structure on $M$ is equivalent to a pair of closed real decomposable $n$-forms $\alpha,\beta$ such that $\alpha\wedge\beta\neq0$ everywhere.
\end{prop}
\begin{proof}
This follows from a null coordinate point view of paracomplex geometry. Here $\alpha$ and $\beta$ are the real and imaginary parts of a nowhere null para-holomorphic $(n,0)$-form on $M$.
\end{proof}

\begin{ex}(cf. \cite{harvey2012}, Note 9.2)\\
Let $M_1$ and $M_2$ be compact manifolds of dimension $n$. Then $M=M_1\times M_2$ has a natural para-Calabi-Yau structure. Moreover, it admits a K\"ahler para-Calabi-Yau structure if and only if both $M_1$ and $M_2$ are parallelizable.
\end{ex}
\begin{ex}(Kodaira-Thurston nilmanifold, see \cite{kodaira1964} and \cite{thurston1976})\\
Let $M=\rr^4/\sim$, where the equivalent relation is given by \[(x_1,x_2,x_3,x_4)\sim(x_1+a,x_2+b,x_3+c,x_4+d-bx_3)\] for any $a,b,c,d\in\zz$. It is easy to see that $M$ is a compact 4-manifold and \[e^1=\ud x_1,\quad e^2=\ud x^2,\quad e^3=\ud x_3,\quad e^4=\ud x_4+x_2\ud x_3\] form a global basis of 1-forms on $M$. Moreover, $\omega=e^1\wedge e^2+e^3\wedge e^4$ is a symplectic form. Let $\alpha=e^1\wedge e^3$ and $\beta=e^2\wedge e^4$. It is clear that $\{\alpha,\beta,\omega\}$ define a K\"ahler para-Calabi-Yau structure on $M$. Similar structures can be constructed on other nilmanifolds and solvmanifolds.
\end{ex}
\begin{ex}
Let us recall the basic setting of SYZ picture and semi-flat mirror symmetry \cite{strominger1996}. Let $N$ be a compact special integral affine manifold of dimension $n$, that is, $N$ admits a collection of coordinate charts such that local coordinate transformations are elements in the group $SL(n,\zz)\ltimes \rr^n$. Let $\{x_1,\dots,x_n\}$ be local coordinates of $N$ and let $\{\xi_1,\dots,\xi_n\}$ be the usual fiber coordinates of $T^*N$. It is obvious that the lattice bundle \[\Lambda^*=\{(x,\xi)\in T^*N:\xi_j\in\zz,~j=1,\dots,n\}\] is well-defined and we may form the quotient $M=T^*N/\Lambda^*$. Clearly, $M$ is a compact $2n$-manifold with a canonical symplectic structure $\omega$ inherited from $T^*N$. Let $\alpha=\ud x_1\wedge\dots\wedge\ud x_n$ and $\beta=\ud\xi_1\wedge\dots\wedge\ud\xi_n$. They are again well-defined because that $N$ is special affine. We see immediately that $\{\alpha,\beta,\omega\}$ define a K\"ahler para-Calabi-Yau structure on $M$.
\end{ex}

Paracomplex structure has made its appearance in string theory in various scenarios. For instance, according to Ro\v cek \cite{rocek1992}, the special geometry involving two commuting complex structures and a paracomplex structure proposed and studied by \cite{gates1984} and \cite{buscher1985} should be treated on an equal footing as usual K\"ahler Calabi-Yau geometry in string compactifications. It was also observed that special para-K\"ahler geometry \cite{cortes2004} arises naturally in Euclidean $\mathcal{N}=2$ supersymmetry among other contexts. However, it seems that the role of para-Calabi-Yau's in this story still awaits to be explored.

In any K\"ahler para-Calabi-Yau space, $\alpha,\beta,\omega$ satisfy the algebraic and differential relations \[\ud\alpha=\ud\beta=\ud\omega=0,~\alpha\wedge\omega=\beta\wedge\omega=0.\] Moreover, Lagrangian foliations, a.k.a polarizations in geometric quantization, arise naturally, since a para-K\"ahler manifold is nothing but a symplectic manifold with two transversal Lagrangian foliations. These two features indicate that para-Calabi-Yau's fit well in the paracomplex analogue of type IIA supergravity with O6-brane in \cite{tseng2014} and type IIA supersymmetry in \cite{lau2014}. It is reasonable to expect that they will play an important role in the future development.

\bibliographystyle{alpha}

\bibliography{C:/Users/Piojo/Dropbox/Documents/Source}

\begin{thebibliography}{CMMS04}

\bibitem[Ada60]{adams1960}
J.F. Adams.
\newblock On the non-existence of elements of {H}opf invariant one.
\newblock {\em Annals of Mathematics}, 72(1):20--104, 1960.

\bibitem[AG86]{abbena1986}
E.~Abbena and A.~Grassi.
\newblock Hermitian left invariant metrics on complex {L}ie groups and
  cosymplectic {H}ermitian manifolds.
\newblock {\em Bollettino della Unione Matematica Italiana-A}, 5(6):371--379,
  1986.

\bibitem[Agr08]{agricola2008}
I.~Agricola.
\newblock Old and new on the exceptional group ${G}_2$.
\newblock {\em Notices of the American Mathematical Society}, 55(8):922--929,
  2008.

\bibitem[AMT09]{alekseevsky2009}
D.V. Alekseevsky, C.~Medori, and A.~Tomassini.
\newblock Homogeneous para-{K}\"ahler {E}instein manifolds.
\newblock {\em Russian Mathematical Surveys}, 64(1):1--43, 2009.

\bibitem[B\"93]{bar1993}
C.~B\"ar.
\newblock Real {K}illing spinors and holonomy.
\newblock {\em Communications in Mathematical Physics}, 154(3):509--521, 1993.

\bibitem[Ban02]{banaru2002}
M.B. Banaru.
\newblock Hermitian geometry of 6-dimensional submanifolds of the {C}ayley
  algebra.
\newblock {\em Sbornik: Mathematics}, 193(5):635--648, 2002.

\bibitem[BG67]{brown1967}
R.B. Brown and A.~Gray.
\newblock Vector cross products.
\newblock {\em Commentarii Mathematici Helvetici}, 42(1):222--236, 1967.

\bibitem[BK94]{banaru1994}
M.B. Banaru and V.F. Kirichenko.
\newblock The {H}ermitian geometry of the 6-dimensional submanifolds of a
  {C}ayley algebra.
\newblock {\em Russian Mathematical Surveys}, 49(1):223--224, 1994.

\bibitem[Bon66]{bonan1966}
E.~Bonan.
\newblock Sur les vari\'et\'es riemanniennes \`a groupe d'holonomie ${G}_2$ ou
  ${S}pin(7)$.
\newblock {\em Comptes Rendus de l'Acad\'emie des Sciences}, 262(1):127--129,
  1966.

\bibitem[Bor53]{borel1953}
A.~Borel.
\newblock Sur la cohomologie des espaces fibr\'es principaux et des espaces
  homogenes de groupes de {L}ie compacts.
\newblock {\em Annals of Mathematics}, 57(1):115--207, 1953.

\bibitem[BS89]{bryant1989}
R.L. Bryant and S.M. Salamon.
\newblock On the construction of some complete metrics with exceptional
  holonomy.
\newblock {\em Duke Mathematical Journal}, 58(3):829--850, 1989.

\bibitem[Bus85]{buscher1985}
T.H. Buscher.
\newblock Quantum corrections and extended supersymmetry in new
  $\sigma$-models.
\newblock {\em Physics Letters B}, 159(2):127--130, 1985.

\bibitem[Cab96]{cabrera1996}
F.M. Cabrera.
\newblock On {R}iemanniam manifolds with ${G}_2$-structure.
\newblock {\em Bollettino della Unione Matematica Italiana-A}, 10(1):99--112,
  1996.

\bibitem[Cab06]{cabrera2006}
F.M. Cabrera.
\newblock ${SU}(3)$-structures on hypersurfaces of manifolds with
  ${G}_2$-structure.
\newblock {\em Monatshefte f\"ur Mathematik}, 148(1):29--50, 2006.

\bibitem[Cal58]{calabi1958}
E.~Calabi.
\newblock Construction and properties of some 6-dimensional almost complex
  manifolds.
\newblock {\em Transactions of the American Mathematical Society},
  87(2):407--438, 1958.

\bibitem[Car94]{cartan1894}
\'E.J. Cartan.
\newblock {\em Sur la structure des groupes de transformations finis et
  continus}, volume 826 of {\em Th\'eses}.
\newblock Nony, 1894.

\bibitem[CFG96]{cruceanu1996}
V.~Cruceanu, P.~Fortuny, and P.M. Gadea.
\newblock A survey on paracomplex geometry.
\newblock {\em Rocky Mountain Journal of Mathematics}, 26(1):83--115, 1996.

\bibitem[CMMS04]{cortes2004}
V.~Cort\'es, C.~Mayer, T.~Mohaupt, and F.~Saueressig.
\newblock Special geometry of {E}uclidean supersymmetry {I}. {V}ector
  multiplets.
\newblock {\em Journal of High Energy Physics}, 2004(03), 2004.

\bibitem[CMS96]{cabrera1996b}
F.M. Cabrera, M.D. Monar, and A.F. Swann.
\newblock Classification of ${G}_2$-structures.
\newblock {\em Journal of the London Mathematical Society}, 53(2):407--416,
  1996.

\bibitem[CS02]{chiossi2002}
S.G. Chiossi and S.~Salamon.
\newblock The intrinsic torsion of ${SU}(3)$ and ${G}_2$ structures.
\newblock In {\em Differential Geometry, Valencia 2001}, pages 115--133. World
  Scientific, 2002.

\bibitem[Eck42]{eckmann1942}
B.~Eckmann.
\newblock Stetige {L}\"osungen linearer {G}leichungssysteme.
\newblock {\em Commentarii Mathematici Helvetici}, 15(1):318--339, 1942.

\bibitem[Eng00]{engel1900}
F.~Engel.
\newblock Ein neues, dem linearen {K}omplexe analoges {G}ebilde.
\newblock {\em Berichte über die Verhandlungen der Königlich Sächsischen
  Gesellschaft der Wissenschaften zu Leipzig. Mathematisch-Physische Klasse},
  52:63--76, 110--139, 1900.

\bibitem[EST06]{etayo2006}
F.~Etayo, R.~Santamar\'ia, and U.R. Tr\'ias.
\newblock The geometry of a bi-{L}agrangian manifold.
\newblock {\em Differential Geometry and Its Applications}, 24(1):33--59, 2006.

\bibitem[Fer87]{fernandez1987}
M.~Fern\'andez.
\newblock An example of a compact calibrated manifold associated with the
  exceptional {L}ie group ${G}_2$.
\newblock {\em Journal of Differential Geometry}, 26(2):367--370, 1987.

\bibitem[FG82]{fernandez1982}
M.~Fern\'andez and A.~Gray.
\newblock Riemannian manifolds with structure group ${G}_2$.
\newblock {\em Annali di matematica pura ed applicata}, 132(1):19--45, 1982.

\bibitem[FI02]{friedrich2002}
T.~Friedrich and S.~Ivanov.
\newblock Parallel spinors and connections with skew-symmetric torsion in
  string theory.
\newblock {\em Asian Journal of Mathematics}, 6(2):303--335, 2002.

\bibitem[FI03]{friedrich2003}
T.~Friedrich and S.~Ivanov.
\newblock Killing spinor equations in dimension 7 and geometry of integrable
  ${G}_2$-manifolds.
\newblock {\em Journal of Geometry and Physics}, 48(1):1--11, 2003.

\bibitem[FLY12]{fu2012}
J.-X. Fu, J.~Li, and S.-T. Yau.
\newblock Balanced metrics on non-{K}\"ahler {C}alabi-{Y}au threefolds.
\newblock {\em Journal of Differential Geometry}, 90(1):81--129, 2012.

\bibitem[FY08]{fu2008}
J.-X. Fu and S.-T. Yau.
\newblock The theory of superstring with flux on non-{K}\"ahler manifolds and
  the complex {M}onge-{A}mp\`ere equation.
\newblock {\em Journal of Differential Geometry}, 78(3):369--428, 2008.

\bibitem[FY15]{fei2015}
T.~Fei and S.-T. Yau.
\newblock Invariant solutions to the {S}trominger system on complex {L}ie
  groups and their quotients.
\newblock {\em Communications in Mathematical Physics}, 338(3):1--13, 2015.

\bibitem[GH80]{gray1980}
A.~Gray and L.M. Hervella.
\newblock The sixteen classes of almost {H}ermitian manifolds and their linear
  invariants.
\newblock {\em Annali di Matematica Pura ed Applicata}, 123(1):35--58, 1980.

\bibitem[GHR84]{gates1984}
S.J. Gates, C.M. Hull, and M.~Rocek.
\newblock Twisted multiplets and new supersymmetric non-linear $\sigma$-models.
\newblock {\em Nuclear Physics B}, 248(1):157--186, 1984.

\bibitem[GM91]{gadea1991}
P.M. Gadea and J.M. Masqu\'e.
\newblock Classification of almost para-{H}ermitian manifolds.
\newblock {\em Rend. Math. Roma}, 11(2):377--396, 1991.

\bibitem[GP04]{goldstein2004}
E.~Goldstein and S.~Prokushkin.
\newblock Geometric model for complex non-{K}\"ahler manifolds with ${SU}(3)$
  structure.
\newblock {\em Communications in Mathematical Physics}, 251(1):65--78, 2004.

\bibitem[Gra66]{gray1966}
A.~Gray.
\newblock Some examples of almost {H}ermitian manifolds.
\newblock {\em Illinois Journal of Mathematics}, 10(2):353--366, 1966.

\bibitem[Gra69a]{gray1969b}
A.~Gray.
\newblock Six dimensional almost complex manifolds defined by means of
  three-fold vector cross products.
\newblock {\em Toh\^oku Mathematical Journal}, 21(4):614--620, 1969.

\bibitem[Gra69b]{gray1969}
A.~Gray.
\newblock Vector cross products on manifolds.
\newblock {\em Transactions of the American Mathematical Society},
  141:465--504, 1969.

\bibitem[Gra71]{gray1971}
A.~Gray.
\newblock Weak holonomy groups.
\newblock {\em Mathematische Zeitschrift}, 123(4):290--300, 1971.

\bibitem[Gur35]{gurevich1935}
G.B. Gurevich.
\newblock Classification of trivectors of rank 8.
\newblock {\em \emph{(in Russian)}, Doklady Akademii Nauk SSSR}, 2:353--355,
  1935.

\bibitem[Gur48]{gurevich1948}
G.B. Gurevich.
\newblock Algebra of trivectors {II}.
\newblock {\em \emph{(in Russian)}, Trudy Seminara po Vektornomu i Tenzornomu
  Analizu}, 6:28--124, 1948.

\bibitem[GYZ03]{gukov2003}
S.~Gukov, S.-T. Yau, and E.~Zaslow.
\newblock Duality and fibrations on ${G}_2$ manifolds.
\newblock {\em Turkish Journal of Mathematics}, 27(1):61--97, 2003.

\bibitem[Her83]{herz1983}
C.S. Herz.
\newblock Alternating 3-forms and exceptional simple {L}ie groups of type
  ${G}_2$.
\newblock {\em Canadian Journal of Mathematics}, 35(5):776--806, 1983.

\bibitem[Hit00]{hitchin2000}
N.J. Hitchin.
\newblock The geometry of three-forms in six dimensions.
\newblock {\em Journal of Differential Geometry}, 55(3):547--576, 2000.

\bibitem[Hit01]{hitchin2001}
N.J. Hitchin.
\newblock Stable forms and special metrics.
\newblock In {\em Global Differential Geometry: The Mathematical Legacy of
  Alfred Gray (Bilbao, 2000)}, volume 288 of {\em Contemporary Mathematics},
  pages 70--89. AMS, 2001.

\bibitem[Hit04]{hitchin2004}
N.J. Hitchin.
\newblock Special holonomy and beyond.
\newblock In {\em Strings and Geometry}, volume~3 of {\em Clay Mathematics
  Proceedings}, pages 159--175. AMS, 2004.

\bibitem[HL82]{harvey1982}
F.R. Harvey and H.B. Lawson.
\newblock Calibrated geometries.
\newblock {\em Acta Mathematica}, 148(1):47--157, 1982.

\bibitem[HL12]{harvey2012}
F.R. Harvey and H.B. Lawson.
\newblock Split special {L}agrangian geometry.
\newblock In {\em Metric and Differential Geometry}, volume 297 of {\em
  Progress in Mathematics}, pages 43--89. Birkh\"ausser, 2012.

\bibitem[Joy96]{joyce1996}
D.D. Joyce.
\newblock Compact {R}iemannian 7-manifolds with holonomy ${G}_2$. {I}.
\newblock {\em Journal of Differential Geometry}, 43(2):291--328, 1996.

\bibitem[JR90]{jensen1990}
G.R. Jensen and M.~Rigoli.
\newblock Neutral surfaces in neutral four-spaces.
\newblock {\em Le Matematiche}, 45(2):407--443, 1990.

\bibitem[Kir73]{kirichenko1973}
V.F. Kirichenko.
\newblock Nearly {K}\"ahler structures that are induced by a triple vector
  product on six-dimensional submanifolds of {C}ayley's algebra.
\newblock {\em Moscow University Mathematics Bulletin}, 28(3):59--65, 1973.

\bibitem[Kir80]{kirichenko1980}
V.F. Kirichenko.
\newblock Classification of {K}\"ahler structures induced by vector triple
  products on six-dimensional submanifolds of the {C}ayley algebra.
\newblock {\em Soviet Mathematics (Iz. VUZ)}, 24(8):36--43, 1980.

\bibitem[Kod64]{kodaira1964}
K.~Kodaira.
\newblock On the structure of compact complex analytic surfaces, {I}.
\newblock {\em American Journal of Mathematics}, 86(4):751--798, 1964.

\bibitem[Kon05]{konderak2005}
J.J. Konderak.
\newblock A {W}eierstrass representation theorem for {L}orentz surfaces.
\newblock {\em Complex Variables}, 50(5):319--332, 2005.

\bibitem[Lib52]{libermann1952}
P.~Libermann.
\newblock Sur les structures presque paracomplexes.
\newblock {\em Comptes Rendus de l'Acad\'emie des Sciences},
  234(26):2517--2519, 1952.

\bibitem[LTY14]{lau2014}
S.-C. Lau, L.-S. Tseng, and S.-T. Yau.
\newblock Non-{K}\"ahler {SYZ} mirror symmetry.
\newblock {\em arXiv: 1409.2765}, 2014.

\bibitem[LWY14]{lin2014}
H.~Lin, B.-S. Wu, and S.-T. Yau.
\newblock Heterotic string compactification and new vector bundles.
\newblock {\em arXiv: 1412.8000}, 2014.

\bibitem[LY87]{li1987}
J.~Li and S.-T. Yau.
\newblock {H}ermitian-{Y}ang-{M}ills connection on non-{K}\"ahler manifolds.
\newblock In {\em Mathematical Aspects of String Theory}, Advanced Series in
  Mathematical Physics Vol. 1, pages 560--573. World Scientific, 1987.

\bibitem[LY05]{li2005}
J.~Li and S.-T. Yau.
\newblock The existence of supersymmetric string theory with torsion.
\newblock {\em Journal of Differential Geometry}, 70(1):143--181, 2005.

\bibitem[LY14]{liu2014}
K.-F. Liu and X.-K. Yang.
\newblock Ricci curvatures on {H}ermitian manifolds.
\newblock {\em arXiv: 1404.2481}, 2014.

\bibitem[Rei07]{reichel1907}
W.~Reichel.
\newblock \"{U}ber die {T}rilinearen {A}lternierenden {F}ormen in 6 und 7
  {V}er\"anderlichen.
\newblock {\em Dissertation, Greifswald}, 1907.

\bibitem[Roc92]{rocek1992}
M.~Rocek.
\newblock Modified {C}alabi-{Y}au manifolds with torsion.
\newblock In {\em Essays on Mirror Manifolds}, pages 480--488. International
  Press, 1992.

\bibitem[Ros13]{rossi2013}
F.A. Rossi.
\newblock D-complex structures on manifolds: cohomological properties and
  deformations.
\newblock {\em Dissertation, Universit\`a degli Studi di Milano-Bicocca}, 2013.

\bibitem[Sch31]{schouten1931}
J.A. Schouten.
\newblock Klassifizierung der alternierenden {G}r\"oszen dritten {G}rades in 7
  {D}imensionen.
\newblock {\em Rendiconti del Circolo Matematico di Palermo}, 55(1):137--156,
  1931.

\bibitem[Str86]{strominger1986}
A.E. Strominger.
\newblock Superstrings with torsion.
\newblock {\em Nuclear Physics B}, 274(2):253--284, 1986.

\bibitem[STW15]{szekelyhidi2015}
G.~Sz\'ekelyhidi, V.~Tosatti, and B.~Weinkove.
\newblock Gauduchon metrics with prescribed volume form.
\newblock {\em arXiv: 1503.04491}, 2015.

\bibitem[SYZ96]{strominger1996}
A.E. Strominger, S.-T. Yau, and E.~Zaslow.
\newblock Mirror symmetry is {T}-duality.
\newblock {\em Nuclear Physics B}, 479(1):243--259, 1996.

\bibitem[Thu76]{thurston1976}
W.P. Thurston.
\newblock Some simple examples of symplectic manifolds.
\newblock {\em Proceedings of the American Mathematical Society},
  55(2):467--468, 1976.

\bibitem[TY14]{tseng2014}
L.-S. Tseng and S.-T Yau.
\newblock Generalized cohomologies and supersymmetry.
\newblock {\em Communications in Mathematical Physics}, 326(3):875--885, 2014.

\bibitem[Whi62]{whitehead1962}
G.W. Whitehead.
\newblock Note on cross-sections in {S}tiefel manifolds.
\newblock {\em Commentarii Mathematici Helvetici}, 37:239--240, 1962.

\bibitem[YS64]{yano1964}
K.~Yano and T.~Sumitomo.
\newblock Differential geometry on hypersurfaces in a {C}ayley space.
\newblock {\em Proceedings of the Royal Society of Edinburgh Section A},
  66(4):216--231, 1964.

\bibitem[Zve65]{zvengrowski1965}
P.~Zvengrowski.
\newblock A 3-fold vector product in $\mathbb{R}^8$.
\newblock {\em Commentarii Mathematici Helvetici}, 40(1):149--152, 1965.

\end{thebibliography}

\end{document}